\documentclass[10pt]{article}
\usepackage{amsmath}
\usepackage{amssymb}
\usepackage{amsthm}
\usepackage{lineno}
\usepackage[mathscr]{eucal}
\theoremstyle{definition}
\newtheorem{definition}{Definition}[section]
\newtheorem{theorem}[definition]{Theorem}
\newtheorem*{theorem*}{Conjecture}

\newtheorem{lemma}[definition]{Lemma}
\newtheorem{corollary}[definition]{Corollary}
\theoremstyle{remark}
\newtheorem{remark}[definition]{Remark}

\newcounter{enumctr}

\newcommand{\N}{\mathbb{N}}
\newcommand{\R}{\mathbb{R}}

\newcommand{\rT}{\mathrm {T}}
\providecommand{\keywords}[1]{\textbf{\textbf{Key words: }} #1}
\begin{document}


\title{\vspace*{-10mm}
Existence, uniqueness and exponential boundedness of global solutions to delay fractional differential equations}
\author{N.D.~Cong\footnote{\tt ndcong@math.ac.vn, \rm Institute of Mathematics, Vietnam Academy of Science and Technology, 18 Hoang Quoc Viet, 10307 Ha Noi, Viet Nam}
\;and\;
H.T.~Tuan\footnote{\tt httuan@math.ac.vn, \rm Institute of Mathematics, Vietnam Academy of Science and Technology, 18 Hoang Quoc Viet, 10307 Ha Noi, Viet Nam}}
\maketitle
\begin{abstract}
Under a mild Lipschitz condition we prove a theorem on the existence and uniqueness of global solutions to delay fractional differential equations. Then, we establish a result on the exponential boundedness for these solutions. 
\end{abstract}

\section{Introduction}
Recently, delay fractional differential equations (DFDEs) have received considerable attentions because they provide mathematical models of real-world  problems in which the fractional rate of change depends on the influence of their hereditary effects, see e.g., \cite{Lakshmikantham,Benchohra,Krol,Yang,Cermak} and the references therein.
The simplest form of DFDEs is 
\begin{equation}\label{add_eq}
\begin{cases}
^{C}D^\alpha_{0+}x(t)=f(t,x(t),x(t-r)),\quad t\in [0,T],\\
x(t)=\phi(t),\quad\forall t\in[-r,0],
\end{cases}
\end{equation}
where $\alpha>0$ is the order of the Caputo fractional derivative $^{C}D^\alpha_{0+}$, and the initial condition $\phi$ is a continuous function on the interval $[-r,0]$, $r,T>0$ are fixed real parameters. For this equation, the first basic and important problem is to show the existence and uniqueness of solutions under some reasonable conditions. It is well known that in the case of ordinary differential equations ($\alpha$ is an integer), under some Lipschitz conditions a delay equation has an unique local solution (see Hale and Lunel \cite[Section 2.2]{Hale}); furthermore, by using continuation property (see  \cite[Section 2.3]{Hale}) one can derive global solutions as well. However, in the fractional case (non-integer $\alpha$) the problem of existence and uniqueness of (local and global) solutions is more complex because of the  \textit{fractional order} feature of the equation which implies history dependence of the solutions, hence, among others, the continuation property is not applicable. 

Abbas~\cite{Abbas} has  discussed the existence of solutions to the DFDE \eqref{add_eq} and used Krasnoselskii's fixed point theorem to show the existence of at least one local solution to \eqref{add_eq}. Y. Jalilian and R. Jalilian  \cite{Jalilian} have proved the existence of a global solution on a finite interval to \eqref{add_eq} for a class of DFDEs by using a fixed point theorem of Leray--Schauder type. Note that in the two papers \cite{Abbas} and \cite{Jalilian} the authors did not derive uniqueness of the solutions.
Yang and Cao \cite{Yang} have dealt with the problem of existence and uniqueness of solution of  a general DFDE, they presented theorems on existence and uniqueness of solutions to the initial values problems for DFDE, however the Lipschitz condition they assume is restrictive and hard to verify because it is a Lipschitz condition with respect to an infinite dimensional (functional) variable which varies in a (big) functional space $B$. Recently, Wang et al.~\cite[Theorem 3.2]{Wang} have formulated and proved uniqueness of global solutions to the equation similar to \eqref{add_eq}  by using the generalized Gronwall inequality. However their proof contain a flaw which make the proof incomplete (see Remark \ref{Remark_1}).

In the investigation of long term behavior of the DFDEs, as in the classical theory of dynamical systems, the understanding of growth rate of the solutions is of basic importance. One needs to know whether the solutions are exponentially bounded so that the theory of Lyapunov exponents as well as the tools of the Laplace transform are applicable to the study of the qualitative behavior of the systems.
Wang et al.~\cite[Theorem 3.3]{Wang} have formulated and proved a theorem on exponential boundedness of solutions of DFDEs; however there are flaws in the proof and the statement of their theorem is false (see Remak \ref{Remark_3} for details).

This paper is devoted to the investigation of the existence, uniqueness and growth rate of global solutions of DFDEs. Namely, we  prove a general theorem on the existence and uniqueness of global solutions to the equation \eqref{add_eq} under a mild Lipschitz condition on $f$ (see Theorem \ref{main_R_1}). An interesting feature of our result is the fact that for the existence and uniqueness of the global solutions of \eqref{add_eq} we do not need to require Lipschitz property of $f$ with respect to the third  (delay) variable of $f$, but only the Lipschitz property of $f$ with respect to the second (non-delay) variable. As concerns the growth rate of solutions of DFDEs, we derive a result on the exponential boundedness of solutions to the equation \eqref{add_eq} (see Theorem \ref{main_R_2}).

The rest of this paper is organized as follows. In Section~\ref{sec.preliminary}, we recall some basic notations of fractional calculus and a lemma concerning the equivalence between a DFDE and a Volterra integral equation. In Section~\ref{sec.main}, we show the existence and uniqueness of global solutions to DFDEs (Theorem \ref{main_R_1}). Finally, in Section~\ref{sec.bound} we establish a result on the exponential boundedness of these solutions (Theorem \ref{main_R_2}). 

\section{Preliminaries}\label{sec.preliminary}
This section is devoted to recalling briefly a framework of DFDEs. We first introduce some notations which are used throughout this paper. Let $\R_{\geq 0}$ be the set of all non-negative real numbers and $\R^d$ be the $d$-dimensional Euclidean space endowed with a norm $\|\cdot\|$. For any $[a,b]\subset [a,\infty)$, let $C([a,b];\R^d)$ be the space of continuous functions $\xi:[a,b]\rightarrow\R^d$ with the sup norm $\|\cdot\|_\infty$, i.e.,
\[
\|\xi\|_\infty:=\sup_{a\le t\le b}\|\xi(t)\|,\quad\forall \xi\in C([a,b];\R^d).
\]
For $\alpha>0$, $[a,b]\subset \R$ and a measurable function $x:[a,b]\rightarrow \R$  such that $\int_a^b|x(\tau)|\;d\tau<\infty$, the Riemann--Liouville integral operator of order $\alpha$ is defined by
\[
(I_{a+}^{\alpha}x)(t):=\frac{1}{\Gamma(\alpha)}\int_a^t(t-\tau)^{\alpha-1}x(\tau)\;d\tau,\quad t\in (a,b],
\]
where $\Gamma(\cdot)$ is the Gamma function. The \emph{Caputo fractional derivative} $^{C\!}D_{a+}^\alpha x$ of a function $x\in AC^m([a,b];\R)$ is defined by
\[
(^{C\!}D_{a+}^\alpha x)(t):=(I_{a+}^{m-\alpha}D^mx)(t),\quad t\in (a,b],
\]
where $AC^m([a,b];\R)$ denotes the space of real functions $x$ which has continuous derivatives up to order $m-1$ on the interval $[a,b]$ and the $({m-1})^{th}$-order derivative $x^{(m-1)}$ is absolutely continuous, $D^m=\frac{d^m}{dt^m}$ is the usual $m^{th}$-order derivative and $m:=\lceil\alpha\rceil$ is the smallest integer larger or equal to $\alpha$. The Caputo fractional derivative of a $d$-dimensional vector function $x(t)=(x_1(t),\cdots,x_d(t))^{\rT}$ is defined component-wise as
\[
(^{C\!}D_{a+}^\alpha x)(t):=(^{C\!}D_{a+}^\alpha x_1(t),\cdots,^{C\!\!}D_{a+}^\alpha x_d(t))^{\rT}.
\]
From now on, we consider only the case $\alpha\in (0,1)$. Let $r$ be an arbitrary positive constant, and  $\phi\in C([-r,0];\R^d)$ be a given
continuous function, we study the delay Caputo fractional differential equations
\begin{equation}\label{IntroEq}
^{C\!}D_{0+}^\alpha x(t)=f(t,x(t),x(t-r)),\quad t\in [0,T],
\end{equation}
with the initial condition 
\begin{equation}\label{Ini_Cond}
x(t)=\phi(t),\quad \forall t\in[-r,0],
\end{equation} 
where $x\in \R^d$, $T>0$ and $f:[0,T]\times\R^d\times\R^d \rightarrow \R^d$ is continuous. 

We also consider the initial condition problem \eqref{IntroEq}-\eqref{Ini_Cond} on the infinite time interval $[-r,\infty)$ as well with the obvious change from finite $T$ to $\infty$.

A function $\varphi(\cdot,\phi)\in C([-r,T];\R^d)$ is called a \emph{solution} of the initial condition problem \eqref{IntroEq}-\eqref{Ini_Cond} over the interval $[-r,T]$ if 
\begin{equation*}
\begin{cases}
^{C\!}D^\alpha_{0+}\varphi(t,\phi)=f(t,\varphi(t,\phi),\varphi(t-r,\phi)),\quad t\in [0,T],\\
\varphi(t,\phi)=\phi(t),\quad \forall t\in[-r,0].
\end{cases}
\end{equation*}
To prove the existence of solutions to the initial condition problem \eqref{IntroEq}-\eqref{Ini_Cond} and to investigate the asymptotic behavior of solutions to this problem we need to convert it into an equivalent delay integral equation with the initial condition \eqref{Ini_Cond}. This is stated in the following lemma.
\begin{lemma}\label{Equivalent_eq}
The function $\varphi\in C([-r,T];\R^d)$ is a solution of the initial condition problem \eqref{IntroEq}-\eqref{Ini_Cond} on the interval $[-r,T]$ if and only if it is a solution of the delay integral equation
\begin{equation}\label{Integral_eq}
x(t)=\phi(0)+\frac{1}{\Gamma(\alpha)}\int_0^t (t-\tau)^{\alpha-1}f(\tau,x(\tau),x(\tau-r))\;d\tau,\quad \forall t\in [0,T]
\end{equation}
with the initial condition
\begin{equation}\label{Ini_integral_eq}
x(t)=\phi(t),\quad\forall t\in[-r,0].
\end{equation}
\end{lemma}
\begin{proof}
Using the same arguments as in the proof of \cite[Lemma 6.2, p.~86]{Kai}.
\end{proof}

\section{Existence and uniqueness of global solutions to delay fractional differential equations}\label{sec.main}
We show that under a mild Lipschitz condition a DFDE has unique global solution. 

\begin{theorem}[Existence and uniqueness of global solutions to DFDEs]\label{main_R_1}
Assume that  $f:[0,T]\times\R^d\times\R^d \rightarrow \R^d$ is continuous and satisfies the following Lipschitz condition with respect to the second variable: there exists a non-negative  continuous function  $L:[0,T]\times \R^d\rightarrow \R_{\geq 0}$ such that
\begin{equation}\label{Lipschitz cond}
\|f(t,x,y)-f(t,\hat x, y)\|\leq L(t,y)\|x-\hat x\|
\end{equation}
for all $t\in [0,T]$, $x,y,\hat x \in \R^d$.
Then, the initial condition problem \eqref{IntroEq}-\eqref{Ini_Cond} has a unique global solution $\varphi(\cdot, \phi)$ on the interval $[-r,T]$.
\end{theorem}
\begin{proof}
According to Lemma \ref{Equivalent_eq}, the equation \eqref{IntroEq} with the initial condition \eqref{Ini_Cond} is equivalent to the initial condition problem \eqref{Integral_eq}-\eqref{Ini_integral_eq}.  

First we consider the case $0<T\leq r$. In this case, the equation \eqref{Integral_eq} has the form
\[
x(t)=\phi(0)+\frac{1}{\Gamma(\alpha)}\int_0^t (t-\tau)^{\alpha-1}f(\tau,x(\tau),\phi(\tau-r))\;d\tau,\quad \forall t\in [0,T].
\]
For this integral equation, by Tisdell~\cite[Theorem 6.4, p. 310]{Tisdell}, there exists a unique solution  on the interval $[0,T]$. Denote that solution by $\xi^*_r$ and put
\begin{equation*} 
\varphi_T(t,\phi):=
\begin{cases}
\phi(t), \quad \forall t\in [-r,0],\\
\xi^*_r(t),\quad \forall t\in [0,T].
\end{cases}
\end{equation*}
Then $\varphi_T(t,\phi)$ is the unique solution of the problem \eqref{Integral_eq}-\eqref{Ini_integral_eq} on $[-r,T]$.

For the case $T>r$, we divide the interval $[0,T]$ into $[0,r]\cup \dots \cup [(k_0-1)r, k_0r]\cup [k_0r,T]$, where $k_0\in \N$ and $0\leq T-k_0r<r$. On the interval $[-r,r]$, using the same arguments as above, we can find a unique solution of the initial condition problem \eqref{Integral_eq}-\eqref{Ini_integral_eq} which is denoted by $\varphi_r(\cdot,\phi)$. We will prove the existence and uniqueness of solution on the interval $[-r,k_0r]$ by induction. Assume that the problem \eqref{Integral_eq}-\eqref{Ini_integral_eq} has a unique solution on the interval $[-r,kr]$ for some $1\leq k<k_0$. We denote that solution by $\varphi_{kr}(\cdot,\phi)$. On $[kr,(k+1)r]$, we define an operator $\mathcal{T}_{(k+1)r,\phi}:C([kr,(k+1)r];\R^d)\rightarrow C([kr,(k+1)r];\R^d)$ as follows:
\begin{align*}
(\mathcal{T}_{(k+1)r,\phi}\xi)(t):=\phi(0)+\frac{1}{\Gamma(\alpha)}\int_0^{kr} (t-\tau)^{\alpha-1}f(\tau,\varphi_{kr}(\tau,\phi),\varphi_{kr}(\tau-r,\phi))\;d\tau\\
+\frac{1}{\Gamma(\alpha)}\int_{kr}^{t}(t-\tau)^{\alpha-1}f(\tau,\xi(\tau), \varphi_{kr}(\tau-r,\phi))\;d\tau,\quad \forall t\in [kr,(k+1)r].
\end{align*}
Let $\beta_k$ be a positive constant satisfying $\beta_k>2\max_{t\in [kr,(k+1)r]}L(t,\varphi_{kr}(t-r,\phi))$. On the space $C([kr,(k+1)r];\R^d)$, we define a new metric $d_{\beta_k}$ by
\[
d_{\beta_k}(\xi,\hat\xi):=\sup_{t\in[kr,(k+1)r]}\frac{\|\xi(t)-\hat\xi(t)\|}{E_{\alpha} (\beta_k t^\alpha)},\quad \forall \xi,\;\hat\xi\in C([kr,(k+1)r];\R^d),
\]
here $E_\alpha:\R\rightarrow \R$ is the Mittag-Leffler function which is defined by 
\[
E_\alpha(z)=\sum_{k=0}^\infty \frac{z^k}{\Gamma(\alpha k+1)},\quad \forall z\in \R.
\]
Then, the space $C(kr,(k+1)r];\R^d)$ equipped the metric $d_{\beta_k}$ is complete. We will show that the operator $\mathcal{T}_{(k+1)r,\phi}$ is contractive on $(C([kr,(k+1)r];\R^d), d_{\beta_k})$. Indeed, for any $\xi,\hat\xi\in C([kr,(k+1)r];\R^d)$ and any $t\in [kr,(k+1)r]$ we have
\begin{align*}
&\|(\mathcal{T}_{(k+1)r,\phi}\xi)(t)-(\mathcal{T}_{(k+1)r,\phi}\hat\xi)(t)\|\\
&\hspace{0.1cm}\leq\; \frac{\displaystyle{\max_{t\in [kr,(k+1)r]}L(t,\varphi_{kr}(t-r,\phi))}}{\Gamma(\alpha)}\int_{kr}^t (t-\tau)^{\alpha-1}\|\xi(\tau)-\hat\xi(\tau)\|\;d\tau\\
&\hspace{0.1cm}\leq\; \frac{\displaystyle{\max_{t\in [kr,(k+1)r]}L(t,\varphi_{kr}(t-r,\phi))}}{\Gamma(\alpha)}\int_{kr}^t (t-\tau)^{\alpha-1}E_\alpha(\beta_k \tau^\alpha)\frac{\|\xi(\tau)-\hat\xi(\tau)\|}{E_\alpha(\beta_k \tau^\alpha)}\;d\tau.
\end{align*}
This implies that
\begin{align*}
&\frac{\|(\mathcal{T}_{kr,\phi}\xi)(t)-(\mathcal{T}_{kr,\phi}\hat\xi)(t)\|}{E_\alpha(\beta_k t^\alpha)}\\
&\hspace{0.1cm}\leq\; \frac{\displaystyle{\max_{t\in [kr,(k+1)r]}L(t,\varphi_{kr}(t-r,\phi))}}{E_\alpha(\beta_k t^\alpha)}d_{\beta_k}(\xi,\hat\xi)\frac{1}{\Gamma(\alpha)}\int_{kr}^t (t-\tau)^{\alpha-1}E_\alpha(\beta_k \tau^\alpha)\;d\tau\\
&\hspace{0.1cm}\leq\; \frac{\displaystyle{\max_{t\in [kr,(k+1)r]}L(t,\varphi_{kr}(t-r,\phi))}}{E_\alpha(\beta_k t^\alpha)}d_{\beta_k}(\xi,\hat\xi)
\frac{1}{\Gamma(\alpha)}\int_0^{t} (t-\tau)^{\alpha-1}E_\alpha(\beta_k \tau^\alpha)\;d\tau\\
&\hspace{0.1cm}\leq\; \frac{\displaystyle{\max_{t\in [kr,(k+1)r]}L(t,\varphi_{kr}(t-r,\phi))}}{E_\alpha(\beta_k t^\alpha)}d_{\beta_k}(\xi,\hat\xi)I^{\alpha}_{0+} \left(^{C\!}D^\alpha_{0+}\left(\frac{E_\alpha(\beta_k t^\alpha)}{\beta_k}\right)\right)\\
&\hspace{0.1cm}\leq\; \frac{\displaystyle{\max_{t\in [kr,(k+1)r]}L(t,\varphi_{kr}(t-r,\phi))}}{\beta_k}d_{\beta_k}(\xi,\hat\xi)
\end{align*}
for all $t\in[kr,(k+1)r]$. Therefore, 
\begin{eqnarray*}
d_{\beta_k}(\mathcal{T}_{(k+1)r,\phi}\xi, \mathcal{T}_{(k+1)r,\phi}\hat\xi)&\leq& \frac{\displaystyle{\max_{t\in [kr,(k+1)r]}L(t,\varphi_{kr}(t-r,\phi))}}{\beta_k}d_{\beta_k}(\xi,\hat\xi)\\
&\leq&\frac{1}{2}d_{\beta_k}(\xi,\hat\xi)
\end{eqnarray*}
for all $\xi,\hat\xi\in C([kr,(k+1)r];\R^d)$. By virtue of Banach fixed point theorem, there exists a unique fixed point $\xi^*_{(k+1)r}$ of $\mathcal{T}_{(k+1)r,\phi}$ in $C([kr,(k+1)r];\R^d)$. Put
\begin{equation}
\varphi_{(k+1)r}(t,\phi):=\begin{cases}
\varphi_{kr}(t,\phi),\quad \forall t\in [-r,kr],\\
\xi^*_{(k+1)r}(t),\quad \forall t\in[kr,(k+1)r].
\end{cases}
\end{equation}
Then, $\varphi_{(k+1)r}(t,\phi)$ is the unique solution of the problem \eqref{Integral_eq}-\eqref{Ini_integral_eq} on $[-r,(k+1)r]$.

Finally, on the interval $[k_0r,T]$, we construct an operator $\mathcal{T}_\phi:C([k_0r,T];\R^d)\rightarrow C([k_0r,T];\R^d)$ by
\begin{align*}
(\mathcal{T}_\phi)(t):=\phi(0)+\frac{1}{\Gamma(\alpha)}\int_0^{k_0r}(t-\tau)^{\alpha-1}f(\tau,\varphi_{k_0r}(\tau,\phi),\varphi_{k_0r}(\tau-r,\phi))\;d\tau\\
+\frac{1}{\Gamma(\alpha)}\int_{k_0r}^{t}(t-\tau)^{\alpha-1}f(\tau,\xi(\tau),\varphi_{k_0r}(\tau-r,\phi))\;d\tau,\quad \forall t\in[k_0r,T].
\end{align*}
Let $\beta_{k_0}$ be a positive number satisfying 
\[
\beta_{k_0}>2\max_{t\in[k_0r,T]}L(t,\varphi_{k_0r}(t-r,\phi)).
\]
By constructing a new metric $d_{\beta_{k_0}}$ on $C([k_0r,T];\R^d)$ as 
$$
d_{\beta_{k_0}}(\xi,\hat\xi):=\sup_{t\in [k_0r,T]}\frac{\|\xi(t)-\hat\xi(t)\|}{E_\alpha(\beta_{k_0} t^\alpha)}
$$ 
and repeating arguments as above, we can show that the operator $\mathcal{T}_\phi$ has a unique fixed point $\xi^*$ on $[k_0r,T]$. Define a function
\begin{equation*}
\varphi_T(t,\phi):=\begin{cases}
\varphi_{k_0r}(t,\phi),\quad \forall t\in[-r,k_0r],\\
\xi^*(t),\quad \forall t\in[k_0r,T].
\end{cases}
\end{equation*}
It is evident that $\varphi_T(\cdot,\phi)$ is the unique solution of the problem \eqref{Integral_eq}-\eqref{Ini_integral_eq} on the interval $[-r,T]$. 
\end{proof}

\begin{corollary}[Existence and uniqueness of global solutions to DFDEs on half line]
If the assumptions of Theorem~\ref{main_R_1} hold on the half line $[-r,\infty)$ then the initial condition problem \eqref{IntroEq}-\eqref{Ini_Cond} has a unique global solution $\varphi(\cdot, \phi)$ on the infinite interval $[-r,\infty)$.
\end{corollary}
\begin{proof}
Suppose that the assumptions of the corollary are satisfied. Note that if $T_1>T_2>0$ are arbitrary two positive numbers. Then the assumptions of Theorem~\ref{main_R_1} hold on both intervals $[-r,T_1]$ and $[-r,T_2]$ implying that we find unique global solutions $\varphi_1$ on $[-r,T_1]$ and $\varphi_2$ on $[-r,T_2]$. Due to uniqueness the function $\varphi_1$ coincides with the function $\varphi_2$ on $[-r,T_2]$.
To complete the proof we establish a new function $\varphi(\cdot,\phi)$ on $[-r,\infty)$ as below
\begin{equation*}
\varphi(t,\phi):=\begin{cases}
\phi(t),\quad\text{if}\;\; t\in [-r,0],\\
\varphi_t(t,\phi),\quad\text{if}\;\; t>0,
\end{cases}
\end{equation*}
where $\varphi_t(\cdot,\phi)$ is the function defined as in the proof of Theorem \ref{main_R_1}. Then, this function is the unique solution to the initial condition problem \eqref{IntroEq}-\eqref{Ini_Cond} on $[-r,\infty)$.
\end{proof}

\begin{remark}\label{Remark_1}
Wang et al. \cite[Theorem 3.2, p. 48]{Wang} have proved a result on the uniqueness of solutions to a DFDE like the problem \eqref{IntroEq}-\eqref{Ini_Cond} under the Lipschitz assumption 
\[
\|f(t,x,y)-f(t,\hat{x},\hat{y})\|\leq L(\|x-\hat{x}\|+\|y-\hat{y}\|), \quad \forall\; t\in\R_{\geq 0},\;\forall \;x,y,\hat{x},\hat{y}\in \R.
\]
 Their approach is based on the generalized Gronwall inequality, and the key point in their proof is the inequality (18) of \cite[p. 49]{Wang}. They deduce this inequality from the inequality (17) of \cite[p. 49]{Wang}. However, this deduction is incorrect due to the fractional nature of the equations. Thus, the proof of Wang et al. is incomplete.  
\end{remark}

\section{Exponential boundedness of solutions to delay fractional differential equations}\label{sec.bound}
For the qualitative theory of DFDEs the study of growth rate of solutions is of basic importance. In this section we show that  solutions of DFDEs are exponentially bounded. 

Let $\phi\in C([-r,0];\R^d)$ be an arbitrary continuous function. We consider the initial condition problem \eqref{IntroEq}-\eqref{Ini_Cond} on the semi real axis $[-r,\infty)$. 
A solution $\varphi(\cdot,\phi)$ of the initial condition problem \eqref{IntroEq}-\eqref{Ini_Cond} is called {\em exponentially bounded} if there exist positive constants $C$ and $\lambda$  such that
\[
\|\varphi(t,\phi)\|\leq C\exp{(\lambda t)},\quad \forall t\geq 0.
\]
The main result in this section is the following theorem on exponential boundedness of solutions of the initial condition problem \eqref{IntroEq}-\eqref{Ini_Cond}.
\begin{theorem}[Exponential boundedness of solutions to delay fractional differential equations]\label{main_R_2}
Assume that $f$ is continuous and satisfies
the following conditions:
\begin{itemize}
\item [(H1)] there exists a positive constant $L$ such that
\[
\|f(t,x,y)-f(t,\hat x,\hat y)\|\leq L(\|x-\hat x\|+\|y-\hat y\|),\quad\forall t\in\R_{\geq 0},x,y,\hat x,\hat y\in\R^d;
\]
\item [(H2)] there exists a constant $\beta>2L$ such that
\[
\sup_{t\geq 0}\frac{\int_0^t (t-\tau)^{\alpha-1}\|f(\tau,0,0)\|\;d\tau}{E_\alpha(\beta t^\alpha)}<\infty.
\]
\end{itemize}
Then the global solution $\varphi(\cdot,\phi)$ on the interval $[-r,\infty)$ of the initial condition problem \eqref{IntroEq}-\eqref{Ini_Cond} is exponentially bounded. More precisely, there exists a constant $C>0$ such that 
\[
\|\varphi(t,\phi)\|\leq C E_\alpha(\beta t^\alpha),\quad \forall t\geq 0.
\]
\end{theorem}
\begin{proof}
We denote by $C_\beta([0,\infty);\R^d)$ the set of all continuous functions
 $\xi\in C([0,\infty);\R^d)$ satisfying the condition
 \[
\|\xi\|_\beta:=\sup_{t\geq 0}\frac{\|\xi(t)\|}{E_\alpha(\beta t^\alpha)} <\infty.
\]
 It is easily seen that $\|\cdot\|_\beta$ is a norm in $C_\beta([0,\infty);\R^d)$ and  
 $(C_\beta([0,\infty);\R^d),\|\cdot\|_\beta)$ is a Banach space.

For any $\phi\in C([-r,0];\R^d)$, we construct a operator $\mathcal{T}_\phi$ on $ (C_\beta([0,\infty);\R^d),\|\cdot\|_\beta)$ as follows:
\begin{align*}
(\mathcal{T}_{\phi}\xi)(t)&:=\phi(0)+\frac{1}{\Gamma(\alpha)}\int_0^{t} (t-\tau)^{\alpha-1}f(\tau,\xi(\tau),\phi(\tau-r))\;d\tau,\quad \forall t\in [0,r],\\
(\mathcal{T}_{\phi}\xi)(t)&:=\phi(0)+\frac{1}{\Gamma(\alpha)}\int_0^{r} (t-\tau)^{\alpha-1}f(\tau,\xi(\tau),\phi(\tau-r))\;d\tau\\
&\hspace{1.33cm}+\frac{1}{\Gamma(\alpha)}\int_r^{t} (t-\tau)^{\alpha-1}f(\tau,\xi(\tau),\xi(\tau-r))\;d\tau,\quad \forall t>r.
\end{align*}
First we consider the case $t\in (0,r]$. In this case we have
\begin{align*}
&\|(\mathcal{T}_{\phi}\xi)(t)\|\leq\; \|\phi(0)\|+\frac{L}{\Gamma(\alpha)}\int_0^t(t-\tau)^{\alpha-1}\Big(\|\xi(\tau)\|\\
&\hspace{3cm}+\|\phi(\tau-r)\|\Big)\;d\tau+\frac{1}{\Gamma(\alpha)}\int_0^t(t-\tau)^{\alpha-1}\|f(\tau,0,0)\|\;d\tau\\
&\hspace{1cm}\leq\; \left(1+\frac{Lr^\alpha}{\Gamma(\alpha+1)}\right)\|\phi\|_\infty+\frac{L}{\Gamma(\alpha)}\int_0^t(t-\tau)^{\alpha-1}E_\alpha(\beta\tau^\alpha)\frac{\|\xi(\tau)\|}{E_\alpha(\beta \tau^\alpha)}\;d\tau\\
&\hspace{3cm}+\frac{1}{\Gamma(\alpha)}\int_0^t(t-\tau)^{\alpha-1}\|f(\tau,0,0)\|\;d\tau,
\end{align*}
where $\|\phi\|_\infty := \sup_{-r\leq s\leq 0}\|\phi(s)\| <\infty$.
This implies that
\begin{align*}
\sup_{t\in[0,r]}\frac{\|(\mathcal{T}_{\phi}\xi)(t)\|}{E_\alpha(\beta t^\alpha)}&\leq \left(1+\frac{Lr^\alpha}{\Gamma(\alpha+1)}\right)\|\phi\|_\infty+\frac{L}{\beta}\|\xi\|_\beta\\
&\hspace{1cm}+\frac{1}{\Gamma(\alpha)}\sup_{t\geq 0}\frac{\int_0^t (t-\tau)^{\alpha-1}\|f(\tau,0,0)\|\;d\tau}{E_\alpha(\beta t^\alpha)}<\infty.
\end{align*}
Next, consider the case $t\geq r$. In this case we have
\begin{align*}
\|(\mathcal{T}_{\phi}\xi)(t)\|\leq&\; \|\phi(0)\|+\frac{L\|\phi\|_\infty}{\Gamma(\alpha)}\int_0^r (t-\tau)^{\alpha-1}\;d\tau\\
&+\frac{1}{\Gamma(\alpha)}\int_0^t(t-\tau)^{\alpha-1}\|f(\tau,0,0)\|\;d\tau\\
&+\frac{L}{\Gamma(\alpha)}\int_0^t (t-\tau)^{\alpha-1}E_\alpha(\beta \tau^\alpha)\frac{\|\xi(\tau)\|}{E_\alpha(\beta \tau^\alpha)}\;d\tau\\
&+\frac{L}{\Gamma(\alpha)}\int_r^t (t-\tau)^{\alpha-1}\|\xi(\tau-r)\|\;d\tau.
\end{align*}
Hence, 
\begin{align*}
\|(\mathcal{T}_{\phi}\xi)(t)\|\leq\;&\|\phi\|_\infty\left(1+\frac{Lt^\alpha}{\Gamma(\alpha+1)}\right)
+\frac{1}{\Gamma(\alpha)}\int_0^t(t-\tau)^{\alpha-1}\|f(\tau,0,0)\|\;d\tau\\
&\hspace{0.5cm}+ \frac{L\|\xi\|_\beta}{\Gamma(\alpha)}\int_0^t(t-\tau)^{\alpha-1}E_\alpha(\beta\tau^\alpha)\;d\tau\\
&\hspace{0.5cm}+\frac{L}{\Gamma(\alpha)}\int_r^t(t-\tau)^{\alpha-1}E_\alpha(\beta (\tau-r)^\alpha)\frac{\|\xi(\tau-r)\|}{E_\alpha(\beta
(\tau-r)^\alpha)}\;d\tau.
\end{align*}
Since $E_\alpha(\cdot)$ is a monotone increasing function on real line, this implies that
\begin{align*}
\frac{\|(\mathcal{T}_{\phi}\xi)(t)\|}{E_\alpha(\beta t^\alpha)}&\leq \|\phi\|_\infty\left(1+\sup_{t\geq r}\frac{Lt^\alpha}{\Gamma(\alpha+1)E_\alpha(\beta t^\alpha)}\right)\\
&\hspace{1cm}+\sup_{t\geq r}\frac{\int_0^t (t-\tau)^{\alpha-1}\|f(\tau,0,0)\|\;d\tau}{\Gamma(\alpha)E_\alpha(\beta t^\alpha)}+\frac{2L}{\beta}\|\xi\|_\beta\\
&\hspace{2cm}<\infty,\quad \forall t\geq r.
\end{align*}
To summarize, the following estimate is true
\[
\sup_{t\geq 0}\frac{\|(\mathcal{T}_{\phi}\xi)(t)\|}{E_\alpha(\beta t^\alpha)}<\infty,\quad
\forall \xi\in (C_\beta([0,\infty);\R^d),\|\cdot\|_\beta).
\]
Thus, $\mathcal{T}_\phi((C_\beta([0,\infty);\R^d),\|\cdot\|_\beta))\subset(C_\beta([0,\infty);\R^d),\|\cdot\|_\beta)$. We now show that the operator $\mathcal{T}_\phi$ is contractive on $(C_\beta([0,\infty);\R^d),\|\cdot\|_\beta)$. Indeed, for any $\xi,\hat\xi\in C_\beta([0,\infty);\R^d),\|\cdot\|_\beta)$, on $[0,r]$ we have the estimate 
\begin{align*}
\|\mathcal{T}_\phi \xi(t)-\mathcal{T}_\phi \hat\xi(t)\|& \leq \frac{L}{\Gamma(\alpha)}\int_0^t(t-\tau)^{\alpha-1}\|\xi(\tau)-\hat\xi(\tau))\|\;d\tau\\
&\leq \frac{L\|\xi-\hat\xi\|_\beta}{\Gamma(\alpha)}\int_0^t (t-\tau)^{\alpha-1}E_\alpha(\beta \tau^\alpha)\;d\tau, \quad\forall t\in(0,r].
\end{align*}
This implies that 
\begin{equation}\label{tam1}
\sup_{t\in [0,r]}\frac{\|\mathcal{T}_\phi\xi(t)-\mathcal{T}_\phi\hat\xi(t)\|}{E_\alpha(\beta t^\alpha)}\leq \frac{L}{\beta}\|\xi-\hat\xi\|_\beta.
\end{equation}
Furthermore, for all $t\geq r$, we have
\begin{align*}
\|\mathcal{T}_\phi\xi(t)-\mathcal{T}_\phi\hat\xi(t)\|\leq\;& \frac{L}{\Gamma(\alpha)}\int_0^t(t-\tau)^{\alpha-1}\|\xi(\tau)-\hat\xi(\tau))\|\;d\tau\\
&\hspace{1cm}+\frac{L}{\Gamma(\alpha)}\int_r^t(t-\tau)^{\alpha-1}\|\xi(\tau-r)-\hat\xi(\tau-r)\|\;d\tau.
\end{align*}
By using the same arguments as above, we obtain
\begin{equation}\label{tam2}
\sup_{t\geq r}\frac{\|\mathcal{T}_\phi \xi(t)-\mathcal{T}_\phi \hat\xi(t)\|}{E_\alpha(\beta t^\alpha)}\leq \frac{2L}{\beta}\|\xi-\hat\xi\|_\beta
\end{equation}
for all $t\geq r$. Combining \eqref{tam1} and \eqref{tam2}, we get
\[
\|\mathcal{T}_\phi \xi-\mathcal{T}_\phi \hat\xi\|_\beta\leq \frac{2L}{\beta}\|\xi(\tau)-\hat\xi(\tau)\|_\beta
\]
for all $\xi,\hat\xi\in (C_\beta([0,\infty);\R^d),\|\cdot\|_\beta)$. Since $\frac{2L}{\beta} <1$, according to the Banach fixed point theorem, there exists a unique fixed point $\xi^*$ of $\mathcal{T}_\phi$ in the space $(C_\beta([0,\infty);\R^d),\|\cdot\|_\beta)$. Put 
\begin{equation*}
\varphi(t,\phi):=\begin{cases}
\phi(t),\quad \forall t\in [-r,0],\\
\xi^*(t),\quad \forall t\in [0,\infty).
\end{cases}
\end{equation*}
It is obvious that $\varphi(\cdot,\phi)$ is the unique global solution of the initial condition problem \eqref{IntroEq}-\eqref{Ini_Cond} on the interval $[-r,\infty)$. From the definition of the space $(C_\beta([0,\infty);\R^d),\|\cdot\|_\beta)$, we can find a constant $C>0$ such that
\[
\|\varphi(t,\phi)\|=\|\xi^*(t)\|\leq C E_\alpha(\beta t^\alpha),\quad \forall t\geq 0.
\]
Due to the asymptotic growth rate of the Mittag-Leffler function $E_\alpha(\beta t^\alpha)$, the solution  $\varphi(\cdot,\phi)$ is exponentially bounded. The proof is complete.
\end{proof}
\begin{remark}\label{Remark_3}
Wang et al.~\cite{Wang} stated a theorem on  exponential boundedness of solutions of delay fractional differential equations. In particular, they  asserted that under the condition \textup{(H1)}, solutions of \eqref{IntroEq} are exponentially bounded for any initial condition $\phi\in C([-r,0];\R^d)$  (see \cite[Theorem 3.3, p.~49]{Wang}). However, their statement is false. For an easy counterexample let us consider the equation 
\begin{equation}\label{Counter_Ex}
\begin{cases}
^{C}D^\alpha_{0+}x(t)=\exp(t^2),\quad t>0,\\
x(t)=x_0\in\R_{\geq 0}\quad\hbox{for}\quad t\in[-r,0],
\end{cases}
\end{equation}
where the fractional order $\alpha\in(0,1)$. In this case, the function $\exp(t^2)$ satisfies the condition $\textup{(H1)}$ above as well as the condition $\textup{(H1)}$ in the statement of \cite[Theorem 3.3, p.~49]{Wang}. By Theorem~\ref{main_R_1}, the equation \eqref{Counter_Ex} has a unique global solution  on $[0,\infty)$, which can be computed explicitly as
\[
\varphi(t,x_0)=x_0+\frac{1}{\Gamma(\alpha)}\int_0^t(t-\tau)^{\alpha-1}\exp(\tau^2)\;d\tau.
\]
It is easily seen that the solution $\varphi(\cdot,x_0)$ is not exponential bounded.
\end{remark}
\section*{Acknowledgement}
This research is funded by the Vietnam National Foundation for
Science and Technology Development (NAFOSTED).

\end{document}